\pdfoutput=1
\documentclass[a4paper,12pt]{amsart}
\usepackage{amsmath,amssymb,amsthm,mathtools,amscd}
\usepackage{array}
\usepackage{enumitem}

\newtheorem{theorem}{Theorem}[section]
\newtheorem{proposition}[theorem]{Proposition}
\newtheorem{lemma}[theorem]{Lemma}
\newtheorem{corollary}[theorem]{Corollary}
\theoremstyle{definition}
\newtheorem{definition}[theorem]{Definition}
\newtheorem{remark}[theorem]{Remark}
\newtheorem{example}[theorem]{Example}

\newcommand{\Mod}{\operatorname{Mod}}
\newcommand{\cH}{\mathcal H}
\newcommand{\Z}{\mathbb Z}
\newcommand{\C}{\mathbb C}
\newcommand{\eps}{\varepsilon}

\title[Planar Calabi--Yau Fillings and Quantum Computation]{Planar Contact
Structures with Calabi--Yau Fillings and Topological Quantum Computation}
\author{Atsuhide Mori}
\address{Department of Mathematics, Osaka Dental University}
\email{mori-a@cc.osaka-dent.ac.jp}

\begin{document}

\begin{abstract}
We study planar openbooks obtained by lifting braids through branched
covers of $D^2$, together with the quantum operations in the Ising
representation.
We give a criterion for the Stein fillings to be Calabi--Yau (CY).
Among positive factorizations of a fixed monodromy, a CY one has minimal
length and its filling minimizes $\chi$ and $b_2$.
Applied to Baykur's recent examples, this gives a planar contact
$3$-manifold with infinitely many non-homeomorphic CY fillings, the cover
there having degree $\ge 6$.
At degree $4$ a single CY filling forces every filling to be CY, and in one
case the filling is unique; at degree $\le 3$ it is unique, and CY under a
mild condition.
Admissible cuts of $D^2$ decompose the openbook into subopenbooks.
Every positive factorization then localizes, and the CY condition holds
exactly when it holds locally. The state space is then a direct sum of
tensor products indexed by the compatible parity choices, with at most two
qubits per factor when the pieces have degree $\le 4$, the same range in
which the CY condition depends only on the monodromy.
For one degree $4$ cover, the points, lines and flags of the two-qubit doily
are realized by its subopenbooks. Fifteen liftable braids there share the
same quantum operation and Stein filling, and are separated only by the
subopenbook systems they admit. A choice of tensor-product structure is thus
carried by the lift, not by the braid group representation, suggesting a link
between contact topology and quantum entanglement.
\end{abstract}

\maketitle

\section{Introduction}
We write \emph{openbook} as one word. We focus on openbooks 
whose page is planar with $d$ boundary components obtained by lifting braids through degree-$d$ 
branched covers of the disk. 
Such openbooks are intended for studying contact topology, 
while the braids also act on quantum state spaces through a unitary representation of the braid group. 
Different braids can represent the same quantum operation while giving different
openbooks. We investigate this freedom particularly in a geometric realization of the Ising representation
$\rho: B_{2d-2}\to \operatorname{PU}(2^{d-2})$. 

A \emph{quasipositive braid} is a product of positive half-twists along arcs joining branch points. 
When each of the factors lifts to a positive Dehn twist, the lifted
factorization determines a Stein filling. Every operation in the Ising representation admits such a realization. 
We call a filling $W$ Calabi--Yau (CY) when $c_1(W)=0$.
For a planar page this is the integral equation
$V^Tu=\mathbf1$, where $V$ is the vanishing-cycle matrix,
and is therefore decided by the factorization alone.
We prove that a CY factorization has
minimal length among positive factorizations of the same monodromy.
Its filling also minimizes $\chi$ and $b_2$ among Stein fillings
associated to that monodromy.

We then cut the covering disk along proper arcs and decompose the
openbook into \emph{subopenbooks} so that the page of the ambient openbook is the
boundary connected sum of the subopenbook pages. 
Then we show that every positive factorization localizes to them,
and the CY condition holds exactly when it holds on each piece.
Accordingly, the state space decomposes as the direct sum of those
of the pieces, over the parity sector choices compatible with the total parity.

This locality is what brings the CY condition into the picture. On a piece of
degree at most four the condition no longer depends on the factorization but
only on the monodromy, and that range is exactly the one in which the tensor
factors carry at most two qubits, the Majorana bilinears exhaust the Pauli
operators and the matchings exhaust the maximal contexts. Beyond it a fixed
monodromy can have infinitely many fillings with the same numerical invariants,
and the description of a piece as a system of qubits is no longer supplied by
the cover. The CY condition is thus not an extra hypothesis taken from
symplectic topology, but the topological side of the range in which the quantum
structure can be named at all.

For the degree-four cover, the six branch points give a two-qubit state
space, the fifteen pairs of branch points label the fifteen nontrivial
Pauli operators, and the fifteen partitions into three pairs correspond to the
maximal contexts via braids which lift to the same CY factorization. Thus, the ambient openbook does
not distinguish contexts. Instead, its subopenbooks do: the points, lines and
flags of the doily are the pairs of branch points, the arc pairings of the
cover, and the cuts separating one pair from the other two.

This separation is sharp. The fifteen braids realizing the fifteen contexts all
act as the identity operation in the Ising representation, and all lift to the
same factorization and hence to the same Stein filling; what distinguishes them
is only which cuts they admit. A representation of the braid group cannot record
this, having a single object, whereas a cut is an object on which the braid
acts. In anyonic terms a cut with its pairings is a fusion tree
\cite{NSSFD2008}, and what the lift adds is which fusion trees a given braid
admits. Read this way, the Ising representation supplies the operations while
the system of subopenbooks supplies the tensor-product structures on which they
act, and the latter is a four-dimensional datum: a cut is a boundary connected
sum decomposition of the Stein filling along which the CY condition is
hereditary.

Our geometry gives examples of quantum subsystems for describing a fixed state.
On the other hand, a two-qubit example shows that even a fixed pairing admits coordinates in which the same vector is a pure
tensor or is not. We ask how such choice of coordinates relate to our cut system.
This viewpoint is motivated by the topos approach to quantum computation in the
companion paper \cite{MoriBayesian2026}, where the same two-qubit context poset
is reached from the side of topos quantum theory.

\section{Branched Covers and Planar Open Books}\label{sec:branched-covers}
A planar surface $P=\Sigma_{0,d}$ is a $(d-1)$-holed disk with boundary 
\[
 \partial P=\partial_0P\sqcup\partial_1P\sqcup\cdots
 \sqcup\partial_{d-1}P\cong \bigsqcup_d S^1
\]
where $\partial_0P$ is the outer boundary component. 
An abstract planar openbook is a pair
$(P,\phi)$ with monodromy class 
\[
 \phi\in\Mod(P\operatorname{rel}\partial P)
 :=\pi_0\{\textrm{diffeomorphisms of $P$ fixing near $\partial P$}\}. 
\]
Filling the boundary of the mapping torus of $\phi$ with the tori $\partial P\times D^2$ yields a closed contact $3$-manifold. 
We say that the contact structure is supported by the openbook. 
The classical Hurwitz theory of simple branched covers gives the following
picture.

\begin{proposition}
[Arc pairing]
\label{prop:paired-cover}
Every degree-$d$ simple branched cover $p:P=\Sigma_{0,d}\longrightarrow D^2$ 
admits pairwise disjoint embedded arcs
$a_1,\ldots,a_{d-1}$ on $D^2$ whose endpoints form the branch set $\Delta\subset \operatorname{int} D^2$ 
and whose preimage contains simple closed curves $\gamma_j\subset p^{-1}(a_j)$ parallel to $\partial_j P$ ($j=1,\dots,d-1$).
With the boundary-induced sheet labeling, the two endpoints of $a_j$
have local monodromy $(0j)$. We also label the branch points $1,\dots, 2d-2$ so that their monodromies are 
\[
 (01), (01), (02), (02), \ldots, (0,d-1), (0,d-1).
\]
\end{proposition}

A braid $\beta\in B_{2d-2}=\Mod((D^2, \Delta) \operatorname{rel} \partial D^2)$ is said to be $p$-\emph{liftable} 
if a representative $f$ of $\beta$ induces a representative $\widetilde f$ of a monodromy class on $P$ satisfying $p\circ\widetilde f=f\circ p$. For a simple arc $a\subset D^2$ joining two branch points without touching the others, taking its neighborhood which is topologically a disk, we can perform the positive half-twist $h_a$ exchanging the endpoints of $a$. For a simple closed curve $\gamma\subset\operatorname{int}P$, we can perform the right-handed Dehn twist $\tau_\gamma$ along $\gamma$. All $p$-liftable braids form the liftable subgroup $L_{p}\subset B_{2d-2}$ with lift homomorphism $\ell_p:L_p \to \Mod(P \operatorname{rel} \partial P)$. 
A \emph{quasipositive} braid is a product of positive half-twists
along possibly circuitous arcs. Such a factorization is \emph{factorwise} $p$-liftable
if every factor lies in $L_p$. Each factor then lifts to a right-handed Dehn twist along a loop component of $p^{-1}(a)$. 

Let $\pi: B_{2d-2}\to S_{2d-2}$ be the endpoint permutation map. 
The half-twists on the arcs $a_j$ are $p$-liftable, and the cubes of the short half-twists between two
successive arcs are $p$-liftable by \cite[Lemma~2.3]{MulazzaniPiergallini2001};
their endpoint permutations are the adjacent transpositions generating $S_{2d-2}$. 
Thus the restriction of $\pi$ to $L_p$ is still onto. 
Particularly, for each $j\in \{1,\dots,d-1\}$ and any two branch points $r, s\in \Delta$, 
we can take $\beta \in L_p$ sending $a_j$ to an arc joining $r$ and $s$ to have
\[
 \ell_p(h_{\beta(a_j)})=\ell_p(\beta h_{a_j} \beta^{-1})
 =\ell_p(\beta)\tau_{\partial_j P}\ell_p(\beta^{-1})
 =\tau_{\partial_j P}.
\]
This implies that the endpoint permutation downstairs and the position of the twist
upstairs are distinct pieces of information.

We say that a positive factorization
\[
\phi=\tau_{\gamma_1}\cdots \tau_{\gamma_m}(:=\tau_{\gamma_1}\circ\cdots\circ \tau_{\gamma_m})
\]
is \emph{allowable} if every $[\gamma_i]$ is non-zero in $H_1(P;\mathbb Z)$. It determines a positive
allowable Lefschetz fibration (PALF), and hence a Stein filling
$W_{\gamma_1,\dots,\gamma_m}$ up to deformation equivalence
\cite{LoiPiergallini2001}.
The boundary of the Stein filling carries the contact structure of complex tangency, 
which is the one supported by the openbook.
Following Ghiggini--Golla--Plamenevskaya \cite[Section~2]{GGP2020}, we fix the winding-number basis 
$e_j=-[\partial_j P]$ of $H_1(P;\mathbb Z)$ ($j\in \{1,\dots,d-1\}$)
and orient any vanishing cycle $\gamma$ so that $[\gamma]$ has binary coordinates, that is, every coordinate is either 
$0$ or $1$. For $J\subset \{1,\dots,d-1\}$ we write $e_J=\sum_{j\in J}e_j$, abbreviated
$e_{ij}=e_{\{i,j\}}$ etc. The class of a vanishing cycle is $e_J$
for the set $J$ of holes it encloses.

We have the four Pauli matrices 
\[X = \begin{pmatrix} 0 & 1 \\ 1 & 0 \end{pmatrix},\quad 
Y = \begin{pmatrix} 0 & -i \\ i & 0 \end{pmatrix},\quad
Z = \begin{pmatrix} 1 & 0 \\ 0 & -1 \end{pmatrix},\quad
I = \begin{pmatrix} 1 & 0 \\ 0 & 1 \end{pmatrix}.
\]
Let $\gamma_1,\gamma_2,\dots,\gamma_{2d-2}$ be the Majorana operators on $(\C^2)^{\otimes{d-1}}$ defined by
\[
\gamma_{2j-1}=Z^{\otimes(j-1)}\otimes X\otimes I^{\otimes(d-1-j)},
\quad
\gamma_{2j}=Z^{\otimes(j-1)}\otimes Y\otimes I^{\otimes(d-1-j)}
\]
provided that $d\geq 2$. The parity $\Pi=i^{d-1}\gamma_1\gamma_2 \cdots \gamma_{2d-2}$ has two eigenvalues $\pm1$ each with $2^{d-2}$-dimensional eigenspaces. 
The Ising representation $\rho: B_{2d-2}\to \operatorname{PU}(2^{d-2})$ is defined on each eigenspace so that the image of the standard braid generators $\sigma_k$ ($k=1,\dots,2d-3$) are 
\[
\rho(\sigma_k)=\frac{1}{\sqrt{2}}(I+\gamma_{k+1}\gamma_k)/\operatorname{U}(1). 
\]

\begin{proposition}[Quasipositive realization]
\label{prop:quasipositive-realization}
For every $\beta\in B_{2d-2}$, there is a factorwise $p$-liftable quasipositive
braid $\beta_+\in L_p$ such that
\[
 \pi(\beta_+)=\pi(\beta),\qquad
 \rho(\beta_+)=\rho(\beta).
\]
Its lifted factorization is allowable and hence determines a
Stein filling.
\end{proposition}

\begin{proof}
Take a quasipositive $p$-liftable braid $b$ realizing $\pi(\beta)$. Then, $\beta'=b^{-1}\beta$ is pure. 
The pure braid group is generated by squared half-twists. In the Ising
representation, the image of a squared half-twist depends
only on its endpoint pair and has order at most two \cite{Bravyi2006,Ahlbrecht2010}. 
Thus, we can realize $\rho(\beta')$ as the image $\rho(\beta'')$ of a product of squares of positive half-twists. 
So we put $\beta_+=b\beta''$. Every half-twist used in the construction of $b$ and $\beta''$
lifts to an inner-boundary twist, so all lifted factors are allowable.
\end{proof}

\section{Positive factorizations and CY condition}\label{sec:cy}
A Stein filling $W$ is \emph{CY} if
\[
 c_1(W)=0\in H^2(W; \Z).
\]
We also call a factorization $\phi=\tau_{\gamma_1}\cdots \tau_{\gamma_m}$
CY when it is positive allowable and its associated Stein filling
$W_{\gamma_1,\dots,\gamma_m}$ is CY.

\begin{definition}[Trivialization class] A \emph{trivialization} of a factorization $\phi=\tau_{\gamma_1}\cdots \tau_{\gamma_m}$ is a class
\[
 u\in H^1(P; \Z)
\]
such that
\[
 u([\gamma_i])=1
 \qquad(i=1,\ldots,m).
\]
\end{definition}

In the winding-number basis, this condition reads $V^Tu=\mathbf1_m$,
where the columns of $V$ are $[\gamma_1],\ldots,[\gamma_m]$.

\begin{proposition}[CY criterion]\label{prop:integral-c1}
A positive allowable factorization is CY if and only if it admits a trivialization.
\end{proposition}

\begin{proof}
Use the standard PALF handle decomposition of $W=W_{\gamma_1,\dots,\gamma_m}$. 
If $P$ has $n$ independent one-handles and the factorization has $m$ vanishing cycles, 
then the relevant cellular cochain complex is
\[
 C^1(W; \Z)\cong\Z^n
 \xrightarrow{\,V^T\,}
 C^2(W; \Z)\cong\Z^m
 \longrightarrow 0
\]
where $V$ is the vanishing-cycle matrix, whose columns are the homology classes $[\gamma_i]$. 
With the planar framing and the chosen orientations, the obstruction
to extending a complex trivialization over each Lefschetz two-handle
is $1$ \cite[proof of Proposition~2.4]{GGP2020}. Thus the first Chern
class is represented by the cochain $\mathbf{1}_m\in\Z^m$, and
\[
 c_1(W)=[\mathbf{1}_m]\in
 H^2(W; \Z)
 =\operatorname{coker}(V^T).
\]
Therefore $c_1(W)=0$ exactly when $\mathbf{1}_m\in\operatorname{im}(V^T)$.
\end{proof}

In particular, the boundary-twist realizations in the proof of 
Proposition~\ref{prop:quasipositive-realization} are CY, with
$u(e_j)=1$ for every $j$.

\begin{lemma}[$m\leq 2\Rightarrow$ CY]\label{lem:short-cy}
Every positive allowable factorization of length at most two on a planar
surface is CY.
\end{lemma}

\begin{proof}
For two nonzero binary vectors, assign value $1$ to one coordinate in their
common support, if it is nonempty. Otherwise, assign value $1$ to one
coordinate in each support. Set all other values to $0$. The cases of zero
or one factor are immediate.
\end{proof}

The family $V_{p,q}$ in Asano--Takahashi's relative-trisection classification
consists of two-cycle PALFs on $\Sigma_{0,4}$ \cite{AsanoTakahashi2026}.
It therefore lies in this automatic CY range.

\begin{remark}[Integrality of CY criterion]
Testing only relations $Vb=0$ can miss torsion. Baykur's generalized lantern
relation on $\Sigma_{0,6}$ \cite[Lemma~3]{Baykur2026} gives a concrete
example. Its six left-hand curves have classes
\[
 e_{23},\quad e_{245},\quad e_{135},\quad e_{14},\quad e_{12},\quad e_{34}.
\]
For the resulting matrix $V$, the unique rational solution of
$V^Tu=\mathbf1_6$ is $u=\frac12(1,1,1,1,0)$, whereas
$V^T(1,1,1,1,0)^T=2\mathbf1_6$. Hence the corresponding $c_1$ is nonzero
of order two. On the other hand,
$\ker V=\mathbb Z(1,0,0,1,-1,-1)$ is annihilated by $\mathbf1_6^T$,
so a relation-only test would miss this torsion. This is also the
$2$-torsion issue in the correction recorded by Wand
\cite[Observation~7.5]{Wand2012} to \cite[Corollary~1.5]{OSS2005}.
\end{remark}

The next lemma is the group-theoretic form of the planar hole and joint-hole multiplicities \cite{PVHM2010}.

\begin{lemma}[Linear and quadratic multiplicities]\label{lem:planar-multiplicities}
The vector $V\mathbf{1}_m$ and the matrix $VV^T$ depend only on the mapping class $\phi$, and not on the chosen 
positive factorization $\phi=\tau_{\gamma_1}\cdots \tau_{\gamma_m}$. 
\end{lemma}

\begin{proof}
For any $\gamma\subset P$, take the set $J$ of holes it encloses, and put 
\[
\mu(\tau_\gamma)=(e_J,\,e_Je_J^T)
\in\mathbb Z^{d-1}\oplus\operatorname{Sym}_{d-1}(\mathbb Z).
\]
To show that $\mu$ respects the Gervais--Luo relations for Dehn twists, 
we check the lantern relation on each $3$-holed disk region, since the others are easy. 
We label its four complementary regions so that the $0$-th contains $\partial_0 P$, and let $J_1,J_2,J_3$ be
the sets of holes lying in the other three ($J_1\sqcup J_2\sqcup J_3\subset \{1,\dots,d-1\}$). 
The four boundary curves of this region then have classes
\[
e_{J_1}+e_{J_2}+e_{J_3},\quad e_{J_1},\quad e_{J_2},\quad e_{J_3},
\]
and the three interior curves appearing in the relation have classes
\[
e_{J_1}+e_{J_2},\quad e_{J_1}+e_{J_3},\quad e_{J_2}+e_{J_3}.
\]
The linear terms agree because
\[
\sum_{k=1}^3e_{J_k}+\sum_{k=1}^3e_{J_k}
=\sum_{1\le k<l\le3}(e_{J_k}+e_{J_l}),
\]
and the quadratic terms agree because
\[
\sum_{k=1}^3e_{J_k}e_{J_k}^T
+\Bigl(\sum_{k=1}^3e_{J_k}\Bigr)\Bigl(\sum_{k=1}^3e_{J_k}\Bigr)^T
=\sum_{1\le k<l\le3}(e_{J_k}+e_{J_l})(e_{J_k}+e_{J_l})^T.
\]
Thus $\mu$ descends to $\Mod(P\operatorname{rel}\partial P)$ and satisfies
\[
\mu(\phi)=\sum_{j=1}^m\mu(\tau_{\gamma_j})=(V\mathbf1_m,\,VV^T).
\]
This proves that both quantities depend only on $\phi$.
\end{proof}

For a cohomology class $u\in H^1(P;\mathbb Z)$, the linear multiplicity
\[
L_\phi(u)=u([\gamma_1])+\cdots+u([\gamma_m])=(V\mathbf{1}_m)^Tu
\]
and the quadratic multiplicity
\[
 Q_\phi(u)=u([\gamma_1])^2+\cdots+u([\gamma_m])^2=u^T(VV^T)u
\]
depend only on $\phi$ and $u$. If $u$ is a trivialization of a positive factorization of $\phi$, we have $L_\phi(u)=Q_\phi(u)=m$. 
 
\begin{proposition}[CY length minimality]\label{prop:cy-minimality}
Suppose that a positive allowable factorization $\phi=\tau_{\gamma_1}\cdots \tau_{\gamma_m}$ 
admits a trivialization $u$. Then any 
other positive factorization $\phi=\tau_{\gamma'_1}\cdots \tau_{\gamma'_{m'}}$ satisfies $m'\geq m$. 
Equality holds if and only if $u$ is also a trivialization of $\tau_{\gamma'_1}\cdots \tau_{\gamma'_{m'}}$.
Consequently every CY factorization of $\phi$ has the same, minimal, length.
\end{proposition}

\begin{proof}
If $m=0$, the inequality is immediate, and $u$ can trivialize the second
factorization only when $m'=L_\phi(u)=0$. Assume henceforth that $m>0$.
For $\phi=\tau_{\gamma'_1}\cdots \tau_{\gamma'_{m'}}$, the Cauchy--Schwarz inequality gives
\[
 L_\phi(u)^2
 =\left(
u([\gamma'_1])+\cdots+u([\gamma'_{m'}])
 \right)^2
 \le m' Q_\phi(u).
\]
Since $L_\phi(u)=Q_\phi(u)=m>0$, this gives $m'\ge m$.
Equality holds exactly when $u([\gamma'_1])=\cdots=u([\gamma'_{m'}])=c$ (a constant). The equalities for $L_\phi$ and $Q_\phi$ then give 
$cm'=m=c^2m'$, so $c=1$. This implies that $u$ is also a trivialization of $\phi=\tau_{\gamma'_1}\cdots \tau_{\gamma'_{m'}}$.
\end{proof}

For Stein fillings, $\operatorname{rank}V=\operatorname{rank}VV^T$
is independent of the positive factorization. The PALF handle complex
gives
\[
 b_1(W_{\gamma_1,\dots,\gamma_m})=(d-1)-\operatorname{rank}V,
 \quad
 b_2(W_{\gamma_1,\dots,\gamma_m})=m-\operatorname{rank}V
\]
and $\chi(W_{\gamma_1,\dots,\gamma_m})=2-d+m$.

\begin{corollary}[CY minimality]\label{cor:cy-filling-minimality}
Let $(Y,\xi)$ be a contact $3$-manifold supported by a positive allowable planar openbook $(P,\phi)$, and suppose one Stein filling $W_0$ arising from a positive factorization of $\phi$ is CY. Then every Stein filling $W$ of $(Y,\xi)$ satisfies
\[
b_1(W)=b_1(W_0),\quad b_2(W)\ge b_2(W_0)\quad (\Rightarrow \chi(W)\ge\chi(W_0)).
\]
If equality holds, then $W$ is also CY. 
\end{corollary}

\begin{proof}
A Stein filling is exact and hence symplectically minimal. Wendl's
theorem deforms its symplectic structure to one supported by a positive
allowable Lefschetz fibration compatible with $(P,\phi)$
\cite{Wendl2010}. The preceding formulas and
Proposition~\ref{prop:cy-minimality} give the inequalities.
In the equality case, that proposition gives a CY factorization.
\end{proof}

In particular, a CY factorization admits no length-decreasing positive substitution, including rational-blowdowns and star-surgery substitutions \cite{EndoMarkVHM2011,Starkston2016}, as long as staying within the planar class.

When $d\leq 3$, Wendl's result implies the uniqueness of Stein filling and it is CY unless three different boundary twists are used simultaneously. However when $d\geq 6$, this is no longer the case. Baykur constructs a fixed monodromy
$\phi\in\Mod(\Sigma_{0,6},\partial\Sigma_{0,6})$
with an infinite sequence of length-five positive factorizations whose Stein fillings
are pairwise non-homeomorphic \cite[Theorem~1 and its proof]{Baykur2026}.
Using this, we have

\begin{corollary}[Infinitely many CY fillings]
\label{cor:baykur-cy}
There exists a contact $3$-manifold supported by a planar openbook with
page $\Sigma_{0,6}$ that admits infinitely many pairwise non-homeomorphic
CY Stein fillings.
\end{corollary}

\begin{proof}
The initial factorization of Baykur has the vanishing-cycle classes $e_{13}$, $e_{23}$, $e_{245}$, $e_{135}$, $e_{14}$, which are evaluated as $1$ by $u=(0,0,1,1,0)$. Thus it is CY. 
Since all subsequent factorizations have the same length five, 
Proposition~\ref{prop:cy-minimality} implies that they are CY.
\end{proof}

When $d=4$, one CY Stein filling forces every Stein filling to be CY. In fact, the possible vanishing-cycle classes are $e_1,e_2,e_3,e_{12},e_{13},e_{23},e_{123}$. A trivialization cannot take value $1$ simultaneously on $e_{12},e_{13},e_{23}$, so one of their multiplicities vanishes. For any competing factorization, invariance of $V\mathbf1_m$ and $VV^T$ shows that the multiplicities can only change by the lantern relation, which adds $k$ factors of each of $e_1,e_2,e_3,e_{123}$, removes $k$ factors of each of $e_{12},e_{13},e_{23}$, and increases the length by $k$. CY length minimality and non-negativity then force $k=0$, and the equality case gives the CY condition.
For integral homology spheres, Oba proved the stronger conclusion that the Stein filling is unique up to diffeomorphism and is either $D^4$ or a Mazur-type manifold \cite{Oba2016}.

\section{Subopenbooks and quantum sectors}\label{sec:subopenbooks}

We return to the settings defined in Section~\ref{sec:branched-covers}.

\begin{definition}[$p$-subopenbook]\label{def:admissible-arc-system}
A finite nonempty family $\mathcal A$ of pairwise disjoint properly embedded
arcs in $D^2\setminus \Delta$ is \emph{$(p,\beta)$-admissible} if
\begin{enumerate}[label=\textup{(\roman*)}]
\item for each component $C$ of the closure of
$D^2\setminus\mathcal A$ and each connected component
$Q$ of its preimage $p^{-1}(C)$, the restriction $q_Q$ of $p$ to $Q$ is of degree $d_Q$ equal to the number of boundary components of $Q$;
\item $\beta$ has a representative $f$ which is the identity on a neighborhood of $\mathcal A$
implying that the lift $\widetilde{f}$ restricts to $\phi_Q\in\Mod(Q\operatorname{rel}\partial Q)$, which is 
the lift of the component braid $\beta_C$ represented by $f|_C$. 
\end{enumerate}
Then, we call $(Q,\phi_Q)$ a $p$-\emph{subopenbook}. 
\end{definition}

The surface $Q$ is planar and admissibility gives $\chi(Q)=2-d_Q$, so $Q\cong\Sigma_{0,d_Q}$ implying that 
$q_Q$ satisfies the assumption of Proposition~\ref{prop:paired-cover}. 

\begin{proposition}[Connected-sum decomposition]
\label{prop:connected-sum}
The decomposition along an admissible arc system expresses the ambient openbook 
as an iterated boundary connected sum, and the contact manifold supported by the openbook as an iterated contact connected sum. Namely, 
\[
 (P,\phi)\cong\natural_Q(Q,\phi_Q),
\qquad
 (Y_\phi,\xi_\phi)\cong\#_Q(Y_{\phi_Q},\xi_{\phi_Q}).
\]
\end{proposition}

\begin{proof}
Put $c=|\mathcal A|$ and take the dual graph $G$ of the lifted
decomposition, with $v$ vertices and $cd$ edges. Since admissibility gives
$\chi(Q)=2-d_Q$ and the degrees over each of the $c+1$ disks sum to $d$,
\[
 \chi(P)+cd=\sum_Q\chi(Q)=2v-(c+1)d,
\]
so $v=cd+1$. Thus, the connectivity implies that $G$ is a tree. Since the monodromy is the
identity near every lifted edge, we can glue the parts along $G$ to obtain $(P,\phi)\cong\natural_Q(Q,\phi_Q)$ and 
$(Y_\phi,\xi_\phi)\cong\#_Q(Y_{\phi_Q},\xi_{\phi_Q})$.
\end{proof}

\begin{theorem}[Localization of positive factorization]\label{thm:positive-localization}
Let $\mathcal A$ be $(p,\beta)$-admissible and let
$\lambda=(\tau_{\gamma_1},\ldots,\tau_{\gamma_m})$ be any positive
Dehn-twist factorization of $\phi$. Then representatives of the vanishing
cycles may be chosen with $\gamma_i\cap p^{-1}(\mathcal A)=\varnothing$ for
all $i$. Consequently every $\gamma_i$ lies in a single $p$-subopenbook $Q$, and
$\lambda$ is, up to commuting factors supported in different components, a
shuffle of positive factorizations $\lambda_Q$ of the restricted monodromies
$\phi_Q$.
\end{theorem}

\begin{proof}
Every component of $p^{-1}(\mathcal A)$ is a proper arc fixed by $\phi$.
By \cite[Proposition~3]{BaykurMondenVHM2017}, the vanishing cycles of a
positive factorization have zero geometric intersection with each such arc,
and the arcs are pairwise disjoint, so representatives of the cycles can be
chosen disjoint from their union. Each cycle therefore lies in one $Q$,
twists in different pieces commute, and restriction gives a positive
factorization $\lambda_Q$ of $\phi_Q$.
\end{proof}

\begin{corollary}[CY locality]
\label{cor:cy-heredity}
Let $\lambda$ be a positive allowable factorization of $\phi$, localized as a
shuffle of the $\lambda_Q$. 
Then $\lambda$ is CY if and only if every
$\lambda_Q$ is CY, 
and in that case $|\lambda'_Q|\ge|\lambda_Q|$ for every $Q$
and every other positive factorization $\lambda'$ of $\phi$ (we write $\phi=[\lambda]=[\lambda']$). 
\end{corollary}

\begin{proof}
The boundary connected-sum tree gives
$H^1(P;\Z)\cong\bigoplus_QH^1(Q;\Z)$; in particular each $\lambda_Q$ is
allowable, and under this splitting the equations $u([\gamma_i])=1$ are
precisely the corresponding equations on the pieces, so
Proposition~\ref{prop:integral-c1} gives the CY equivalence.
Theorem~\ref{thm:positive-localization} also localizes $\lambda'$, and
Proposition~\ref{prop:cy-minimality} applied to each $\phi_Q$ gives the
inequalities.
\end{proof}

\begin{example}[Separating pairs]\label{ex:separating-pairs}
Take $\beta=h_{a_1}^{k_1}\cdots h_{a_{d-1}}^{k_{d-1}}$ with all $k_j\ge1$, so that
$\phi=\tau_{\partial_1P}^{k_1}\cdots\tau_{\partial_{d-1}P}^{k_{d-1}}$, and let
$\mathcal A$ separate the $d-1$ pairs from one another. The cut is
admissible: the branched pieces are annuli carrying $\tau^{k_j}$ and the
remaining pieces are unbranched disks. The filling of an annulus with
$\tau^{k}$ is the $A_{k-1}$ Milnor fibre, so
Proposition~\ref{prop:connected-sum} identifies $W$ with the boundary
connected sum of the $A_{k_j-1}$ Milnor fibres and $(Y_\phi,\xi_\phi)$ with
the connected sum of the links $(L(k_j,k_j-1),\xi_{\mathrm{std}})$. The
factorization is CY, with $b_1=0$, $b_2=\sum_j(k_j-1)$ and
$\chi=2-d+\sum_jk_j$.
\end{example}

For the Majorana operators of Section~\ref{sec:branched-covers} with
\[
 \Pi=i^{d-1}\gamma_1\cdots\gamma_{2d-2},\qquad
 \cH^{2^{d-2}}_{\eta}=\ker(\Pi-\eta I)\quad(\eta=\pm1),
\]
we put $\Gamma_{rs}=i\gamma_r\gamma_s$, so that $\Gamma_{rs}^2=I$ and a half-twist with endpoints $r,s$ satisfies
$\rho(h_a^2)=\Gamma_{rs}/\operatorname{U}(1)$ \cite{Bravyi2006,Ahlbrecht2009}. 
A squared half-twist is a pure braid whose quantum image is a nontrivial Pauli
operator. Let $\mathsf P_n\cong\mathbb F_2^{2n}$ be the projective
$n$-qubit Pauli group. The maximal contexts generated by Pauli operators are its Lagrangian subspaces.

\begin{theorem}[Paired cut and sector reduction]
\label{thm:subopenbook-sector}
When $d\ge3$, take $\mathcal{A}$ as a single arc separating the first $2d-4$
branch points from the last pair, provided that a representative of $\beta\in L_p$ is the 
identity near $\mathcal A$. The branched pieces are 
$Q\cong\Sigma_{0,d-1}$ over the first block and an annulus over the last pair. 
Then, for each $\eps\in\{\pm1\}$, there exists an isometry
\[
 E_\eps:=\ker(\Gamma_{2d-3,2d-2}-\eps I)\cong\cH^{2^{d-3}}_{\eta\eps}
\]
intertwining the Ising actions of braids supported on the first $2d-4$ points,
while braids on the last pair act on $E_\eps$ by scalar phases. 
Put $S=[\Gamma_{2d-3,2d-2}]\in \mathsf P_{d-2}$ to identify the projective Pauli group on $E_\eps$ with the symplectic
quotient $S^\perp/\langle S\rangle\cong \mathsf P_{d-3}$, where 
the quotient map restricts to a bijection between the maximal contexts of
$\mathsf P_{d-2}$ containing $S$ and those of $\mathsf P_{d-3}$.
\end{theorem}

\begin{proof}
The two blocks have branched sheets $\{0,\ldots,d-2\}$ and $\{0,d-1\}$,
all other components being unbranched disks. From
\[
\Pi=(i^{d-2}\gamma_1\dots\gamma_{2d-4})\Gamma_{2d-3,2d-2},
\] 
even products of Majorana operators on the first $2d-4$ generators 
act on $E_\eps$ as on $\cH^{2^{d-3}}_{\eta\eps}$, and the half-twist on the last pair acts up to phase
as $\exp(\frac{\pi i}{4}\Gamma_{2d-3,2d-2})$ \cite{Bravyi2006}, hence as a scalar on $E_\eps$. The
Pauli operators preserving $E_\eps$ form $S^\perp$, and the kernel of their projective
restriction $\langle S\rangle$ gives the symplectic quotient. Its
Lagrangian subspaces correspond exactly to the Lagrangian subspaces of
$\mathsf P_{d-2}$ containing $S$.
\end{proof}

\begin{remark}[Stabilization] 
Particularly, when the removed pair carries a single positive half-twist, it is a positive Hopf band and restoring it is a planar positive
Hopf stabilization which does not change the contact structure; for a higher
power the cut splits off an $A_{k-1}$ Milnor fibre instead, as in
Example~\ref{ex:separating-pairs}. On the state space the same cut turns
$\Gamma_{2d-3,2d-2}$ from a variable into the constant $\eps$. The same passage
from a variable to a constant occurs in \cite{MoriBayesian2026}, where a
determined context is excluded from later prediction.
\end{remark}

Iterate the above cut until every pair of
branch points is separated from the others. All
$d-1$ branched pieces are then annuli and the sectors are lines. Writing
$L_t$ for the joint eigenline of the paired operators
$\Gamma_{1,2},\ldots,\Gamma_{2d-3,2d-2}$ with eigenvalues $t_1,\ldots,t_{d-1}$, we have 
\[
 \cH^{\,2^{d-2}}_{\eta}=\bigoplus_{t_1\cdots t_{d-1}=\eta}L_t.
\]
Performing only some of these cuts, we can distribute the pairs into $k$ nonempty blocks
$I_1\sqcup\cdots\sqcup I_k$. Its branched piece over a block $I$ is $Q_I\cong\Sigma_{0,|I|+1}$.
We call such a cut, together with an arc pairing of each $Q_I$ in the sense of
Proposition~\ref{prop:paired-cover}, a \emph{paired decomposition}, and identify
two of them when they induce the same partition and the same matching of
$\Delta$. The pairing is already implicit in the operators $\Gamma_{2j-1,2j}$
above, and recording it makes a paired decomposition determine a matching of
$\Delta$, namely the arc pairing of $p$ formed by all the blocks together.

\begin{proposition}[Sectorwise tensor decomposition]
\label{prop:sectorwise-tensor}
Let $\cH^{\,2^{|I|-1}}_{\eps}$ be the parity-$\eps$ state space of $Q_I$.
Then
\[
 \cH^{\,2^{d-2}}_{\eta}\cong
 \bigoplus_{\substack{\eps_I\in\{\pm1\}\\ \prod_I\eps_I=\eta}}
 \bigotimes_{I}\cH^{\,2^{|I|-1}}_{\eps_I},
\]
the summand of $(\eps_I)$ being the sum of the lines $L_t$ with
$\eps_I=\prod_{j\in I}t_j$, and a braid supported away from the cuts preserves
each summand and acts there as the tensor product of the actions of its
component braids.
\end{proposition}

\begin{proof}
Since the block parities $\Pi_I=\prod_{j\in I}\Gamma_{2j-1,2j}$ commute and multiply
to the total parity, their joint eigenspaces are the summands, and the
dimensions match because $2^{k-1}\prod_I2^{|I|-1}=2^{d-2}$. 
On each summand, the even products of Majorana operators generate mutually commuting full matrix algebras, 
which gives the tensor factors. A braid supported in one block acts by such products for
that block.
\end{proof}

A block of $|I|$ pairs carries $|I|-1$ qubits, so each cut trades a qubit for
a sign, and if no block contains more than three pairs then every tensor
factor is a space of at most two qubits. By
Corollary~\ref{cor:cy-heredity} a CY factorization restricts to a CY
(minimal) factorization of each $Q_I$. In that range the CY condition loses its
dependence on the factorization.

\begin{proposition}[CY at small degree]
\label{prop:cy-small-degree}
Let $Q$ be a $p$-subopenbook with $d_Q\le4$. Then whether a positive allowable
factorization of $\phi_Q$ is CY depends only on $\phi_Q$. Consequently, if every
branched piece of a paired decomposition has degree at most four, equivalently
if every tensor factor carries at most two qubits, then the CY condition is a
property of the decomposition.
\end{proposition}

\begin{proof}
For $d_Q=2$ every vanishing cycle has class $e_1$, so $u(e_1)=1$ is a
trivialization. For $d_Q=3$ let $n_1,n_2,n_{12}$ be the multiplicities of the
classes $e_1,e_2,e_{12}$. Then $V\mathbf1=(n_1+n_{12},\,n_2+n_{12})$ and
$n_{12}$ is the off-diagonal entry of $VV^T$, so
Lemma~\ref{lem:planar-multiplicities} makes all three invariants of $\phi_Q$,
while a trivialization exists exactly when one of them vanishes. For $d_Q=4$
this is the lantern argument at the end of Section~\ref{sec:cy}. The last
assertion is Corollary~\ref{cor:cy-heredity}.
\end{proof}

\begin{corollary}[Discrete cuts are CY]
\label{cor:discrete-cy}
If the discrete paired decomposition is $(p,\beta)$-admissible, then every
positive factorization of $\phi$ is CY, and $\rho(\beta)$ is diagonal in the
basis of the lines $L_t$.
\end{corollary}

\begin{proof}
Every branched piece is an annulus, and
Theorem~\ref{thm:positive-localization} localizes every positive factorization
of $\phi$ along the cut, so Corollary~\ref{cor:cy-heredity} and the case
$d_Q=2$ above give the first claim. The summands of
Proposition~\ref{prop:sectorwise-tensor} are here the lines $L_t$, which
$\beta$ preserves.
\end{proof}

Every paired decomposition thus regroups the same lines $L_t$; what
it adds is the tensor product inside each sector. That is genuine extra
data, since entanglement is relative to a choice of tensor-product structure
\cite{ZanardiLidarLloyd2004}, which the lines alone do not determine. In the
anyonic description of the Ising representation these decompositions are the
fusion trees \cite{NSSFD2008}; the difference appears only once the braid is
taken into account, and is the subject of Remark~\ref{rem:fusion-tree} below.

\section{The two-qubit doily}\label{sec:six-point}

Now we focus on the case where $d=4$.
The fifteen pairs $\{r,s\}\subset \Delta$ then correspond bijectively
to the fifteen nontrivial operators $[\Gamma_{rs}]$ of $\mathsf P_2$, and
$\Gamma_{rs}$ commutes with $\Gamma_{tu}$ exactly when the two pairs are disjoint. The
maximal contexts are therefore the $6!/(2^33!)=15$ perfect matchings
$rs\mid tu\mid vw$ of $\Delta$, while a single operator generates a minimal
nontrivial context. Pairs and matchings are the points and the lines of the
two-qubit doily $W(3,2)$
\cite{BravyiTerhalLeemhuis2010,PlanatSaniga2008}, the generalized quadrangle
of order $(2,2)$: fifteen points, fifteen lines, three points on each line,
three lines through each point, forty-five flags, and
$\operatorname{Aut}(W(3,2))\cong\operatorname{Sp}(4,2)\cong S_6$.
Conjugation by $\rho(\sigma_k)$ acts on the Majorana operators as a signed
transposition, so the action of $\rho(L_p)$ on $\mathsf P_2$ factors through
$\pi|_{L_p}$ and is onto $\operatorname{Sp}(4,2)$; since $\rho(L_p)$ also
contains every $[\Gamma_{rs}]$, the $p$-liftable braids realize the full projective
two-qubit Clifford group \cite{Ahlbrecht2009,Ahlbrecht2010}.

\begin{lemma}[Contexts as arc pairings]\label{lem:context-pairing}
For every maximal context $M=rs\mid tu\mid vw$, there exists $g\in L_p$ such that
the arcs $a_j^M=g(a_j)$ pair $\Delta$ according to $M$ and 
satisfy $\ell_p(h_{a_j^M})=\tau_{\partial_j P}$ $(j=1,2,3)$. 
They form an arc pairing of the same cover $p$ in the sense of
Proposition~\ref{prop:paired-cover}. The squared half-twists
$h_{a_j^M}^2$ are then commuting pure braids in $L_p$. 
The lift factorization of their product has the $M$-independent form
\[
 \lambda_0=(\tau_{\partial_1P},\tau_{\partial_1P},\tau_{\partial_2P},
 \tau_{\partial_2P},\tau_{\partial_3P},\tau_{\partial_3P}).
\]
\end{lemma}

\begin{proof}
The restriction of $\pi$ to $L_p$ is onto, so $g$ may be chosen with $\pi(g)$
carrying the ordered matching $12\mid34\mid56$ to $M$. Then the computation of Section~\ref{sec:branched-covers} gives
$\ell_p(h_{g(a_j)})=\ell_p(g)\tau_{\partial_j P}\ell_p(g)^{-1}=\tau_{\partial_j P}$. 
This implies that $a_j^M$ has the same local monodromy $(0j)$ at both endpoints. 
Expanding the product of the squares gives $\lambda_0$.
\end{proof}

This is Example~\ref{ex:separating-pairs} with $d=4$ and $k_1=k_2=k_3=2$ yielding 
\[
 W_{\lambda_0}\cong\natural^3T^*S^2 
 \quad\textrm{with}\quad
 (\partial W_{\lambda_0},\xi_{\lambda_0})\cong\#^3(\mathbb{RP}^3,\xi_{\mathrm{std}}).
\]
This object itself does not distinguish contexts, namely, by
Lemma~\ref{lem:context-pairing} the lift is $\lambda_0$ for all fifteen of
them. This is forced, since every vanishing cycle of $\lambda_0$ is parallel
to a boundary component and fixed by every boundary-fixed mapping
class, while $\pi(L_p)=S_6$ permutes the
fifteen matchings transitively. What the lemma does provide is that the
whole doily is available inside the single presentation $(P,[\lambda_0],p)$: its
points are the pairs of branch points, its lines are the arc pairings of $p$,
and its flags are the arc pairings with a distinguished pair. Thus, the contexts
are distinguished instead by the cuts. 

\begin{proposition}[Paired decompositions are the subcontexts]
\label{prop:doily-subopenbook}
Let $M$ be a maximal context and write $\Gamma_1,\Gamma_2,\Gamma_3$
for its operators, so that $\Gamma_1\Gamma_2\Gamma_3$ is a scalar and
$\langle M\rangle\cong\mathbb F_2^2$. The paired decompositions whose matching
is $M$ are the five partitions of its three pairs, and sending such a
decomposition to the subgroup of $\langle M\rangle$ generated by its block
parity is a bijection onto the set of subgroups of
$\langle M\rangle$, precisely, onto the set of subcontexts of $M$ in the following table.
\begin{center}
\normalfont\small
\begin{tabular}{@{}llll@{}}
\hline
partition of the pairs & subcontext & branched pieces & sectors\\
\hline
$\{1,2,3\}$ & $0$ & $\Sigma_{0,4}$ & two qubits\\
$\{j\}\sqcup\{l,m\}$ & $\langle \Gamma_j\rangle$ & annulus, $\Sigma_{0,3}$
 & one qubit, two signs\\
$\{1\}\sqcup\{2\}\sqcup\{3\}$ & $\langle M\rangle$ & three annuli
 & four lines\\
\hline
\end{tabular}
\end{center}
\end{proposition}

\begin{proof}
The pairing arcs are disjoint, so any prescribed partition of them is
realized by disjoint proper arcs of the disk, and the resulting decomposition
is admissible: a piece carrying one pair has branched sheet orbit of size two
and lifts to an annulus, a piece carrying two pairs has orbit of size three
and lifts to $\Sigma_{0,3}$, and the other components are unbranched disks. 
The parity operator of a block $I$ is $\prod_{j\in I}\Gamma_j$. 
Since the product $\Gamma_1\Gamma_2\Gamma_3$ is a scalar, the coarsest
partition gives the trivial subgroup, a partition with singleton $\{j\}$
gives $\langle \Gamma_j\rangle$ from both of its blocks, and the discrete
partition gives all of $\langle M\rangle$; these exhaust the five subgroups
of $\mathbb F_2^2$, and the assignment is injective.
\end{proof}

For the middle row, Theorem~\ref{thm:subopenbook-sector} with $d=4$
identifies the maximal contexts of the reduced one-qubit system with the
three lines of $W(3,2)$ through the point $[\Gamma_j]$. For $M=12\mid34\mid56$
and $j=3$ these are
\[
 12\mid34\mid56,\qquad 13\mid24\mid56,\qquad 14\mid23\mid56,
\]
reducing to the three one-qubit Pauli measurements; on $E_\eps$ one has
$\Gamma_{34}=\eta\eps \Gamma_{12}$ because $\Gamma_{12}\Gamma_{34}\Gamma_{56}=\eta I$. Thus the
fifteen points, the fifteen lines and the forty-five flags of $W(3,2)$ are
realized by the pairs of branch points, the arc pairings of $p$, and the
$15\times3$ paired decompositions with a singleton block, the residue of a
point being the three arc pairings of the $p$-subopenbook attached to it. The
same incidence poset, with its forty-five inclusions and its fifteen
five-element building blocks, is the poset of contexts of $\mathsf P_2$ in
\cite{MoriBayesian2026}, a building block there being the lattice of
Proposition~\ref{prop:doily-subopenbook}. Thus,
we can see that the doily is carried by the system of $p$-subopenbooks of one
particular presentation.

\begin{remark}[Which subcontexts survive]
\label{rem:admissible-subcontexts}
Admissibility can be thought of as a condition on the braid, so a general
$\beta\in L_p$ realizes only part of the lattice of
Proposition~\ref{prop:doily-subopenbook}.
If the discrete decomposition is $(p,\beta)$-admissible, then
Proposition~\ref{prop:sectorwise-tensor} makes $\beta$ preserve each of the four
lines $L_t$, that is, $\rho(\beta)$ is diagonal in the basis determined by $M$. 
No such decomposition is admissible once $\rho(\beta)$ fails to be
diagonal. Both cases occur. A braid represented by a word in the
half-twists $h_{a_j^M}$ is the identity near any cut separating the pairing
arcs, so every partition is admissible for it, and indeed each
$\rho(h_{a_j^M})$ is a function of $\Gamma_j$ and hence diagonal. If instead
$a$ joins two branch points $r,s$ lying in different pairs, then $\{r,s\}$ meets
each of those pairs in one point, so $\Gamma_{rs}$ anticommutes with two of
$\Gamma_1,\Gamma_2,\Gamma_3$; neither $\rho(h_a^2)=\Gamma_{rs}/\operatorname{U}(1)$
nor $\rho(h_a)$ is then diagonal, and the decomposition separating those two
pairs is not admissible for $h_a$. Thus the admissible paired decompositions of
$\beta$ form an invariant of $\beta$, and the next proposition shows that it is
strictly finer than both $\rho(\beta)$ and $\ell_p(\beta)$.
\end{remark}

\begin{proposition}[Separation by cuts]
\label{prop:lift-separates}
For each maximal context $M$ put
$\beta_M=h_{a_1^M}^2h_{a_2^M}^2h_{a_3^M}^2\in L_p$. Then, for all fifteen $M$,
\[
 \rho(\beta_M)=I/\operatorname{U}(1),
 \qquad
 \ell_p(\beta_M)=[\lambda_0],
\]
so that the quantum operation, the openbook and the Stein filling are the same.
Nevertheless the discrete paired decomposition of a maximal context $M'$ is
$(p,\beta_M)$-admissible if and only if $M'=M$. Hence the admissible paired
decompositions are determined neither by $\rho(\beta)$ nor by $\ell_p(\beta)$.
\end{proposition}

\begin{proof}
Since $\rho(h_{a_j^M}^2)=\Gamma_j/\operatorname{U}(1)$ and
$\Gamma_1\Gamma_2\Gamma_3$ is a scalar, the product $\rho(\beta_M)$ is a scalar,
and $\ell_p(\beta_M)=[\lambda_0]$ is Lemma~\ref{lem:context-pairing}. For the
last claim, each $h_{a_j^M}^2$ is the Dehn twist along the boundary of a regular
neighborhood of $a_j^M$, and these three curves are disjoint. Such a product of
twists fixes a proper arc up to isotopy exactly when the arc is disjoint from
all three curves, so condition (ii) says that no pair of $M$ is separated by
$\mathcal A$. A cut realizing the discrete decomposition of $M'$ separates the
three pairs of $M'$, so every pair of $M$ lies inside a pair of $M'$, which
forces $M'=M$.
\end{proof}

The content here is not that $\rho$ has a kernel, which is clear, but that the
kernel carries the datum that $\rho$ does not. The fifteen braids differ in
which tensor-product structures are available to them, and availability is a
condition on the lift.

\begin{remark}[Comparison with fusion trees]
\label{rem:fusion-tree}
The combinatorial shadow of a paired decomposition is familiar from the anyonic
description of the Ising representation. There $\cH^{2^{d-2}}_\eta$ is the space
of conformal blocks of $2d-2$ Ising anyons in a disk, the eigenvalue of
$\Gamma_{rs}$ is the fusion channel of the pair $\{r,s\}$, and a grouping of the
branch points together with a pairing inside each group is a fusion tree, whose
fusion basis consists of the lines $L_t$ \cite{NSSFD2008}. A fusion tree is
however a choice of basis, made independently of the braid. The condition on
$\beta$ that it can express is that $\rho(\beta)$ be diagonal in the associated
basis, which admissibility implies. The converse fails as strongly as it can: by
Proposition~\ref{prop:lift-separates} each $\rho(\beta_M)$ is a scalar and hence
diagonal in every basis, while the discrete paired decomposition of $M'$ is
admissible for $\beta_M$ only when $M'=M$. Admissibility is a condition on the
lift, and accordingly a paired decomposition is a datum on the filling: a
boundary connected sum decomposition along which positive factorizations
localize and the CY condition is hereditary.
\end{remark}

\begin{remark}[Why six points] 
The choice $d=4$ is the exact one.
The degree-$d$ paired cover carries $d-2$ qubits. Four branch points give
one qubit, but complementary bilinears become projectively equal after fixing
parity, so endpoint pairs are redundant. At six branch points
$\binom62=15=4^2-1$, the bilinears are exactly all nontrivial two-qubit
Paulis, and $S_6\cong\operatorname{Sp}(4,2)$ supplies their full symplectic
symmetry. The match fails from three qubits onward: with eight branch points
only $\binom82=28$ of the $4^3-1=63$ nontrivial Paulis are bilinears, the
$105$ perfect matchings form part of the $135$ maximal contexts, and $S_8$ is
a proper subgroup of $\operatorname{Sp}(6,2)$. This agrees with the Ising
braid image being the full Clifford group for two qubits but a proper
subgroup for more \cite{Ahlbrecht2009,Ahlbrecht2010}.
Since a tensor decomposition into qubits is named by a maximal context with
ordered generators, it follows that beyond degree four a branched piece is a
space of dimension $2^{d_Q-2}$ whose description as $d_Q-2$ qubits the cover no
longer supplies.
A threshold appears on the topological side as well. For $d\le3$ Wendl's
theorem makes the Stein filling unique, and for $d=4$ the CY condition is a
property of the monodromy by Proposition~\ref{prop:cy-small-degree}, whereas
the fillings of Corollary~\ref{cor:baykur-cy} at $d=6$ are infinitely many and
pairwise non-homeomorphic with the same $b_1$, $b_2$ and $\chi$, so that there
the CY condition retains only its numerical content. Cutting into pieces of
degree at most four is therefore not a convenience: it is the range in which
the bilinears exhaust the Pauli operators and the CY condition still constrains
the filling itself. We do not know that the two thresholds coincide, the
quantum one being crossed at degree five and the topological one by degree six.
\end{remark}

A paired decomposition fixes the sectors but not the identification
of a sector with a tensor product of qubits, and that identification is what
separability (i.e., non-entanglement property) refers to. No cut supplies it,
since each cut trades a qubit for a sign and $k$ blocks leave only $d-1-k$
qubits; the full $d-2$ are present only before any cut, and there as a
dimension rather than as a tensor product. Calling $\cH^{2^{d-2}}_\eta$ a space
of $d-2$ qubits thus presupposes a further choice, namely a maximal context
with ordered generators.

\begin{example}[Different qubit coordinates]
\label{ex:pairing-coordinates}
For the pairing $12\mid34\mid56$ put $\Gamma_1=\Gamma_{12}$, $\Gamma_2=\Gamma_{34}$, $\Gamma_3=\Gamma_{56}$,
so that $\Gamma_1\Gamma_2\Gamma_3=\eta I$ on $\cH^4_\eta$, and let $v_{lm}$
$(l,m\in\{0,1\})$ be unit vectors spanning the four joint eigenlines, with
$\Gamma_1v_{lm}=(-1)^l v_{lm}$ and $\Gamma_2v_{lm}=(-1)^m v_{lm}$. Writing $f_0,f_1$ for
the standard basis of $\C^2$, the ordered pair $(\Gamma_1,\Gamma_2)$ determines the
isomorphism
\[
 T_{12}:\cH^4_\eta\longrightarrow\C^2\otimes\C^2,
 \qquad v_{lm}\longmapsto f_l\otimes f_m,
\]
whereas the ordered pair $(\Gamma_1,\eta \Gamma_3)$ determines $T_{13}$. Since
$\eta \Gamma_3=\Gamma_1\Gamma_2$ has eigenvalue $(-1)^{l+m}$ on $v_{lm}$, one has
$T_{13}(v_{lm})=f_l\otimes f_{l\oplus m}$ with $\oplus$ modulo two.
For the fixed vector $\psi=\frac{1}{\sqrt{2}}(v_{00}+v_{10})$ this gives
\[
 T_{12}(\psi)=\tfrac1{\sqrt2}(f_0+f_1)\otimes f_0,
 \qquad
 T_{13}(\psi)=\tfrac1{\sqrt2}(f_0\otimes f_0+f_1\otimes f_1),
\]
a pure tensor in the first coordinates and a tensor of rank two in the
second, although the pairing and the chosen eigenvectors are unchanged.
\end{example}

The example displays, in the smallest possible case, the datum that no cut
supplies: the pairing and the eigenlines are fixed throughout, and only the
choice of ordered generators changes.
This leads to the question of which changes of coordinates are compatible
with admissible cuts and CY factorizations, and of how the separability of a
fixed state depends on those choices. We also ask how much of $\beta\in L_p$ is
determined by its admissible paired decompositions, and which sublattices of
Proposition~\ref{prop:doily-subopenbook} occur. Behind both is the view that a
predictive theory carries a limit of applicability, given by a selection of
contexts which shifts as data is obtained \cite{MoriBayesian2026}. The tensor
decomposition on which such a selection operates is there a hypothesis; here it
is produced by the geometry, a paired decomposition being what makes the
available operators split blockwise, while the selection itself is the set of
decompositions admissible for a given braid.

\section*{Declaration of competing interest}
The author declares that he has no known competing financial interests or
personal relationships that could have appeared to influence the work
reported in this paper.

\section*{Data availability}
No data was used for the research described in the article.

\section*{Declaration of generative AI and AI-assisted technologies in the manuscript preparation process}
The author utilized ChatGPT (OpenAI) and Claude (Anthropic) as supportive tools during the preparation of this work. These tools were employed to assist with literature search, verify computations, facilitate mathematical discussions that helped refine key arguments, and improve stylistic clarity. All definitions, proofs, original research direction, and final text remain entirely the author's own, and the author assumes full responsibility for the contents of this publication.


\end{document}